\documentclass[12pt]{article}
\usepackage{amsmath,amsfonts,amssymb,amsthm,fancyhdr,bm,mathtools}
\usepackage[hidelinks]{hyperref}
\usepackage{fullpage}
\usepackage[dvipsnames]{xcolor}
\usepackage[capitalize]{cleveref}
\usepackage[color=Purple!50!white]{todonotes}
\usepackage{tikz}
\usetikzlibrary{backgrounds}
\usepackage[numbers,sort]{natbib}

\newtheorem{thm}{Theorem}[section]
\newtheorem{lem}[thm]{Lemma}
\newtheorem{prop}[thm]{Proposition}

\newtheorem{conj}[thm]{Conjecture}
\newtheorem{question}[thm]{Question}
\newtheorem*{claim}{Claim}
\theoremstyle{definition}
\newtheorem{Def}[thm]{Definition}

\crefname{equation}{equation}{equations}
\crefname{thm}{Theorem}{Theorems}
\crefname{section}{Section}{Sections}
\crefname{figure}{Figure}{Figures}
\crefname{prop}{Proposition}{Propositions}

\newcommand\ab[1]{\lvert#1\rvert}
\newcommand{\inv}{^{-1}}

\newcommand\F{\mathcal{F}}
\newcommand\K{\mathcal{K}}
\newcommand\G{\mathcal{G}}
\newcommand\h{\mathcal{H}}
\DeclareMathOperator\ex{ex}

\mathchardef\mhyphen="2D
\newcommand\polyrem{\delta_{\mathrm{poly\mhyphen rem}}}
\newcommand\linrem{\delta_{\mathrm{lin\mhyphen rem}}}
\newcommand\popedge{\delta_{\mathrm{pop\mhyphen edge}}}

\title{Minimum degree and the graph removal lemma}
\author{Jacob Fox\thanks{Department of Mathematics, Stanford University, Stanford, CA 94305, USA. Email: \url{jacobfox@stanford.edu}. Research supported by a Packard Fellowship and by NSF award  DMS-185563.} \and Yuval Wigderson\thanks{Department of Mathematics, Stanford University, Stanford, CA 94305, USA. Email: \url{yuvalwig@stanford.edu}. Research supported by NSF GRFP Grant DGE-1656518.}}
\date{}

\begin{document}
\maketitle
\begin{abstract}
	The clique removal lemma says that for every $r \geq 3$ and $\varepsilon>0$, there exists some $\delta>0$ so that every $n$-vertex graph $G$ with fewer than $\delta n^r$ copies of $K_r$ can be made $K_r$-free by removing at most $\varepsilon n^2$ edges. The dependence of $\delta$ on $\varepsilon$ in this result is notoriously difficult to determine: it is known that $\delta\inv$ must be at least super-polynomial in $\varepsilon\inv$, and that it is at most of tower type in $\log \varepsilon \inv$.

	We prove that if one imposes an appropriate minimum degree condition on $G$, then one can actually take $\delta$ to be a linear function of $\varepsilon$ in the clique removal lemma. Moreover, we determine the threshold for such a minimum degree requirement, showing that above this threshold we have linear bounds, whereas below the threshold the bounds are once again super-polynomial, as in the unrestricted removal lemma.

	We also investigate this question for other graphs besides cliques, and prove some general results about how minimum degree conditions affect the bounds in the graph removal lemma.
\end{abstract}

\section{Introduction}
One of the deepest results in extremal graph theory is the triangle removal lemma of Ruzsa and Szemer\'edi \cite{MR519318}, as well as its extension to the graph removal lemma, proved independently by Alon--Duke--Lefmann--R\"odl--Yuster \cite{MR1251840} and F\"uredi \cite{MR1404036}. Loosely, this result says that if a large graph $G$ contains ``few'' copies of a fixed graph $H$, then it can be made $H$-free by deleting ``few'' edges. The formal statement is as follows.

\begin{thm}\label{thm:graph-removal}
	Let $H$ be a graph on $h$ vertices. For every $\varepsilon>0$, there exists a $\delta>0$ such that the following holds. If $G$ is an $n$-vertex graph with fewer than $\delta n^h$ copies of $H$, then one can remove at most $\varepsilon n^2$ edges from $G$ to make it $H$-free. 
\end{thm}
Despite its simple statement, the graph removal lemma is a deep result, with many applications in number theory, computer science, and graph theory. For more on the removal lemma and its history, we refer to the survey \cite{MR3156927}.

Many important questions surrounding the graph removal lemma remain open. Most notably, the correct bound for $\delta$ in terms of $\varepsilon$ is unknown. Formally, let $\delta(\varepsilon,H)$ denote the maximum\footnote{Note that this maximum is attained (i.e.\ that one can write ``maximum'' rather than ``supremum''), because we require \emph{fewer than} $\delta \ab{V(G)}^h$ copies of $H$, but allow deleting \emph{at most} $\varepsilon\ab{V(G)}^2$ edges.} $\delta$ such that the following holds for every graph $G$: if $G$ has fewer than $\delta \ab{V(G)}^h$ copies of $H$, then $G$ can be made $H$-free by removing at most $\varepsilon \ab{V(G)}^2$ edges. 
The best lower bound on $\delta(\varepsilon,H)$, due to Fox \cite{MR2811609}, shows that $\delta(\varepsilon,H) \geq 1/T(O_h(\log \frac 1 \varepsilon))$, where $T$ is the tower function, recursively defined by $T(0)=1$ and $T(x)=2^{T(x-1)}$ for $x\geq 1$. For the upper bound, Alon \cite{MR1945375} (extending \cite{MR519318} and \cite{MR932119}) showed that if $H$ is not bipartite\footnote{If $H$ is bipartite, then $\delta(\varepsilon,H)=\varepsilon^{\Theta_H(1)}$, i.e.\ the removal lemma has polynomial bounds \cite{MR1945375}. The removal lemma is somewhat degenerate in case $H$ is bipartite, since in this case the entire problem reduces to counting copies of bipartite graphs in dense graphs, which can be done with the method of K\H ovari--S\'os--Tur\'an \cite{MR65617}. This is closely related to a famous conjecture of Erd\H os--Simonovits and Sidorenko, see e.g.\ \cite{MR2738996} for details.}, then $\delta(\varepsilon,H) \leq \varepsilon^{\Omega_H(\log \frac 1 \varepsilon)}$ as $\varepsilon \to 0$. In particular, if $H$ is not bipartite, then $1/\delta(\varepsilon,H)$ must be at least super-polynomial in $1/\varepsilon$. Even in the first non-trivial case, of $H=K_3$, these bounds remain the best known results. 

Another important class of results in extremal graph theory concerns structural results implied by minimum degree conditions. Notable examples include Dirac's theorem \cite{MR47308} on the existence of Hamiltonian cycles, its extension by Koml\'{o}s, S\'{a}rk\"{o}zy, and Szemer\'{e}di \cite{MR1682919} to powers of Hamiltonian cycles, and the Andr\'asfai--Erd\H os--S\'os theorem \cite{MR340075} on when a $K_r$-free graph is $(r-1)$-partite. 

In this paper, we study a natural minimum-degree version of the graph removal lemma. Formally, let us define $\delta(\varepsilon,H;\gamma)$ to be the maximum $\delta \in [0,1]$ such that every $n$-vertex graph with fewer than $\delta n^h$ copies of $H$ and minimum degree at least $\gamma n$ can be made $H$-free by deleting at most $\varepsilon n^2$ edges. Here, $\gamma \in [0,1]$ is some constant which we think of as fixed, and we are interested in the behavior of $\delta(\varepsilon,H;\gamma)$ as $\varepsilon \to 0$. We remark that this function is non-decreasing in $\gamma$, and that setting $\gamma=0$ recovers the earlier definition of $\delta(\varepsilon,H)$, that is, $\delta(\varepsilon,H;0)=\delta(\varepsilon,H)$.

Our main results show that $\delta(\varepsilon,K_r;\gamma)$ is linear in $\varepsilon$ if $\gamma>\frac{2r-5}{2r-3}$, but that it is super-polynomial in $\varepsilon$ if $\gamma< \frac{2r-5}{2r-3}$. Formally, we first prove the following theorem, which asserts that $\delta(\varepsilon,K_r;\gamma)$ is linear in $\varepsilon$ for $\gamma > \frac{2r-5}{2r-3}$.
\begin{thm}\label{thm:above-threshold}
	For every $r \geq 3$, there exists $\mu_r>0$ such that for all $\alpha, \varepsilon>0$,
	\[
		\delta\left(\varepsilon, K_r; \frac{2r-5}{2r-3}+\alpha\right) \geq \mu_r \alpha \varepsilon,
	\]
	meaning that $\delta(\varepsilon,K_r;\gamma)$ is linear in $\varepsilon$ for all $\gamma>\frac{2r-5}{2r-3}$.
\end{thm}
Our next result implies that below the threshold $\frac{2r-5}{2r-3}$, the $K_r$ removal lemma must have super-polynomial bounds. In fact, we are able to relate the behavior of the restricted removal function $\delta(\varepsilon,K_r;\gamma)$ to that of the unrestricted \emph{triangle} removal function $\delta(\varepsilon,K_3)$; since this function is known to have super-polynomial bounds, we conclude the same for $\delta(\varepsilon,K_r;\gamma)$. Formally, we prove the following result. 
\begin{thm}\label{thm:below-threshold}
	For every integer $r \geq 3$ and every $\alpha>0$, there exists some $C=C({r,\alpha})>0$ such that for every $\varepsilon>0$, 
	\[
		\delta\left(\varepsilon,K_r;\frac{2r-5}{2r-3}- \alpha\right) \leq  \delta (C\varepsilon, K_3).
	\]
	In particular, $\delta(\varepsilon, K_r; \gamma)\inv$ is super-polynomial in $\varepsilon\inv$ for fixed $\gamma<\frac{2r-5}{2r-3}$.
\end{thm}
Somewhat surprisingly, our technique does not enable us to upper-bound $\delta(\varepsilon, K_r;\gamma)$ in terms of $\delta(\varepsilon, K_r)$ for $\gamma<\frac{2r-5}{2r-3}$. Because of this, to prove super-polynomial bounds on the restricted $K_r$ removal function, we must use that such bounds are known for the unrestricted triangle removal function.

The results in \cref{thm:above-threshold,thm:below-threshold} tell us that $\frac{2r-5}{2r-3}$ is a minimum degree threshold for the $K_r$ removal lemma: below this threshold, the removal lemma has super-polynomial bounds, whereas above it we have linear bounds. We can formalize this notion of threshold as follows.
\begin{Def}\label{def:removal-thresholds}
	Let $H$ be a graph. We define the \emph{linear removal threshold} of $H$ to be
	\[
		\linrem(H) = \inf \{ \gamma \in [0,1]:\text{there exists } \mu>0\text{ so that }\delta(\varepsilon,H; \gamma) \geq \mu \varepsilon\text{ for all }\varepsilon \in (0,1)\}.
	\]
	Similarly, we define the \emph{polynomial removal threshold} of $H$ to be
	\[
		\polyrem(H) = \inf \{ \gamma \in [0,1]:\text{there exists } \mu>0\text{ so that }\delta(\varepsilon,H; \gamma) \geq \mu \varepsilon^{1/\mu}\text{ for all }\varepsilon \in (0,1)\}.
	\]
\end{Def}
These thresholds measure the weakest possible minimum degree condition one can impose in order to have, respectively, linear and polynomial bounds in the graph removal lemma for $H$. In this language, \cref{thm:above-threshold,thm:below-threshold} can be rephrased as saying that
\[
	\linrem(K_r) = \polyrem(K_r) = \frac{2r-5}{2r-3}.
\]
In addition to determining the linear and polynomial removal thresholds for $K_r$, we also prove some results about $\linrem(H)$ and $\polyrem(H)$ for more general classes of graphs, and make a number of conjectures about the relationship between these thresholds and other well-known thresholds in extremal graph theory. We refer to \cref{sec:conclusion} for more details. 

We remark that other versions of the graph removal lemma have been studied under certain minimum degree--like assumptions, such as in \cite{MR3653100,MR3063153}. More generally, there is a long line of work on how the numerical dependencies in the removal lemma (and in Szemer\'edi's regularity lemma) can be improved under certain assumptions about the host graph, e.g.\ \cite{MR2455594,MR3926281,MR2341924,MR2815610,MR3943496,MR3585030,MR3145742}.

The rest of the paper is organized as follows. In the next section, we prove \cref{thm:above-threshold}. In \cref{sec:below-threshold}, we prove \cref{thm:below-threshold} by exhibiting a specific graph of high minimum degree and poor $K_r$ removal properties. We end with some concluding remarks, where we generalize these results and study $\linrem(H)$ and $\polyrem(H)$ for general graphs $H$, and discuss the connections this problem has to the chromatic and homomorphism thresholds of graphs. For the sake of clarity of presentation, we omit all floor and ceiling signs whenever they are not crucial.

\section{Above the threshold: the proof of Theorem \ref{thm:above-threshold}}\label{sec:above-threshold}
In this section, we prove \cref{thm:above-threshold}. Unlike all known proofs of the full graph removal lemma, our proof uses only simple averaging arguments to find a small set of edges, each of which lies in many copies of $K_r$. We then show that removing all these edges deletes all copies of $K_r$ in $G$. Crucially, these averaging arguments only work because of our minimum degree assumption; as shown by \cref{thm:below-threshold}, they cannot possibly work if the minimum degree is below $(\frac{2r-5}{2r-3}-\alpha)n$ for any fixed $\alpha>0$.

Here is a restatement of \cref{thm:above-threshold}, restated to indicate what exactly we will prove in this section.
\begin{thm}\label{thm:above-threshold-extended}
	For every $r \geq 3$, there exists $\mu_r>0$ such that the following holds for all $\alpha, \varepsilon>0$. Let $G$ be an $n$-vertex graph with minimum degree at least $(\frac{2r-5}{2r-3}+\alpha)n$, and suppose that $G$ contains at most $(\mu_r \alpha \varepsilon) n^r$ copies of $K_r$. Then $G$ can be made $K_r$-free by deleting at most $\varepsilon n^2$ edges. 
\end{thm}

We will need the following simple fact from calculus (or basic algebra).

\begin{lem}\label{lem:calc-inequality}
	For any $x \geq 4$, we have that
	\[
		x \frac{2x-5}{2x-3} \geq x-2 + \frac 25.
	\]
\end{lem}
\begin{proof}
	Differentiating shows that the function $f(x)=x \frac{2x-5}{2x-3} - (x-2)$ is monotonically increasing, so its value for all $x \geq 4$ is lower-bounded by its value at $x=4$, and $f(4) = \frac 25$. 
\end{proof}
Our main technical result is the following lemma, which says that if $G$ has minimum degree at least $(\frac{2r-5}{2r-3}+\alpha)n$, then every $K_r$ in $G$ contains a ``popular'' edge, namely an edge lying in $\Omega_r(\alpha n^{r-2})$ copies of $K_r$. 

\begin{lem}\label{lem:key-lemma}
	Let $r \geq 3$ and $\alpha>0$. If $G$ is an $n$-vertex graph with minimum degree at least $(\frac{2r-5}{2r-3} + \alpha)n$, then every $K_r$ in $G$ contains an edge which lies in at least $c_r \alpha n^{r-2}$ copies of $K_r$, for some constant $c_r>0$ depending only on $r$.
\end{lem}
\begin{proof}
	Fix a copy of $K_r$ in $G$, and let its vertices be $v_1,\dotsc,v_r$. For $i \in [r]$, let $V_i = N(v_i)$ denote the neighborhood of $v_i$. Note that by the minimum degree condition, we have that $\ab{V_i} \geq (\frac{2r-5}{2r-3} + \alpha)n$ for each $i$. We will prove the following claim by induction.

	\begin{claim} 
		For each integer $0 \leq t \leq r-3$, there exists a set $S_t \subseteq [r]$ of size $\ab{S_t} = r-t$ and at least $c_{r,t} n^{t}$ copies of $K_t$ whose vertices lie in $\bigcap_{i \in S_t}V_i$, for some constant $c_{r,t}>0$. 
	\end{claim}
	\begin{proof}[Proof of claim]
		The base case $t=0$ is trivial, since we simply take $S_0 = [r]$ and $c_{r,0}=1$. 
		Inductively, suppose we have found a set $S_t$ with the desired properties, for $t \leq r-4$. Let $Q$ be a copy of $K_t$ with vertices in $\bigcap_{i \in S_t} V_i$.  Since every vertex in $Q$ has degree at least $(\frac{2r-5}{2r-3} + \alpha)n$, there are at most $(\frac{2}{2r-3} - \alpha)n < \frac{2}{2r-3}n$ vertices not adjacent to any given vertex in $Q$. Thus, the common neighborhood of $Q$ has size at least
		\[
			m \coloneqq n - t \left(\frac{2}{2r-3} n \right) = \frac{2r - 3 -2t}{2r-3} n.
		\]
		Note that for any vertex $v \in V(G)$ and for any set of $m$ vertices in $G$, the number of edges between $v$ and this $m$-set is at least
		\[
			\left(\frac{2r-5}{2r-3} + \alpha \right)n - (n-m) > \frac{2r-5 - 2t}{2r-3}n = \frac{2r-5-2t}{2r-3-2t}m.
		\]
		Now consider an auxiliary bipartite graph $B_t$, whose first part consists of $S_t$, whose second part consists of $m$ arbitrary common neighbors of the vertices in $Q$, and where a vertex $v$ in the second part is adjacent to a vertex $i$ in the first part if $v \in V_i$. By the computation above, each vertex in the first part of $B_t$ has degree at least $\frac{2r-5-2t}{2r-3-2t}m$.
		The first part of $B_t$ has $r-t$ vertices. Hence, the average degree in the second part of $B_t$ is at least $(r-t)\frac{2r-5-2t}{2r-3-2t}$, which is at least $r-t-2+ \frac25$, by \cref{lem:calc-inequality} applied to $x=r-t$ and using the fact that $t \leq r-4$, which implies that $x\geq 4$. By Markov's inequality, at least $m/5$ vertices in the second part of $B_t$ have degree at least $r-t-1$. Therefore, there are at least $(c_{r,t} n^{t})(m/5)$ choices of a clique $Q$ contained in $\bigcap_{i \in S_t} V_i$, and a common neighbor of $Q$ that lies in at least $r-t-1$ of the sets $V_i$ for $i \in S_t$. Hence, by the pigeonhole principle, for at least $c_{r,t} n^t m/ (5 (r-t))$ of these choices, the same subset of $S_t$ of order $r-t-1$ is used. We let $S_{t+1}$ be this subset, and let
		\[
			c_{r,t+1} = \frac{c_{r,t}}{5(r-t)} \frac mn = \frac{(2r-3-2t)}{5(2r-3)(r-t)}c_{r,t},
		\]
		so that there are at least $c_{r,t+1}n^{t+1}$ choices of a $K_{t+1}$ whose vertices lie in $\bigcap_{i \in S_{t+1}}V_i$. This completes the proof of the claim.
	\end{proof}

	To conclude, we actually run the same argument for $t=r-3$, except that we need to be more careful about keeping track of the parameter $\alpha$. Let $S = S_{r-3}$ be the set given by the claim for $t=r-3$, and let $c = c_{r,r-3}$. Let $Q$ be a $K_{r-3}$ whose vertices lie in $\bigcap_{i \in S} V_i$. Let $B$ be the bipartite graph whose first part consists of three vertices, labeled by the elements of $S$, and whose second part consists of $m$ common neighbors of the vertices in $Q$, where
	\[
		m = n - (r-3) \left( \frac{2}{2r-3} - \alpha \right) n = \left(\frac{3}{2r-3} + (r-3) \alpha\right)n.
	\]
	By the same argument as above, each vertex in the first part of $B$ has degree at least
	\begin{align*}
		\left(\frac{2r-5}{2r-3} + \alpha \right)n - (n-m) &= \left( \frac{1}{2r-3} + (r-2) \alpha \right) n \\
		&= \frac{\frac{1}{2r-3}+(r-2)\alpha}{\frac{3}{2r-3}+(r-3) \alpha}m\\
		&=\frac{1+(r-2)(2r-3)\alpha}{3+(r-3)(2r-3) \alpha} m\\
		&\geq \left(\frac 13+c' \alpha\right)m,
	\end{align*}
	for some constant $c'>0$ depending only on $r$. Therefore, the average degree in the second part of $B$ is at least $1+3c' \alpha$. By Markov's inequality, this implies that at least $\frac 32 c'\alpha m$ vertices in this part have at least two neighbors in the first part. Hence, there are at least $c n^{r-3}\cdot \frac 32 c' \alpha m$ choices of a $K_{r-3}$ and a vertex in its common neighborhood which lies in at least two of the three sets $V_i$ for $i \in S$. By the pigeonhole principle, there is some $\{i,j\} \subset S$ such that $V_i \cap V_j$ contains at least $c_r \alpha n^{r-2}$ copies of $K_{r-2}$, where
	\[
		c_r = \frac{cc'}{2} \frac mn = \frac{cc'}{2}\left(\frac{3}{2r-3} + (r-3) \alpha\right) \geq \frac{3cc'}{2(2r-3)}.
	\]
	Therefore, the edge $\{v_i,v_j\}$ in our original $K_r$ lies in at least $c_r \alpha n^{r-2}$ copies of $K_r$.
\end{proof}

Using \cref{lem:key-lemma}, we can prove \cref{thm:above-threshold-extended}, and thus \cref{thm:above-threshold}.
\begin{proof}[Proof of \cref{thm:above-threshold-extended}]
	Let $\mu_r = c_r/\binom r2$, where $c_r$ is the constant from \cref{lem:key-lemma}, and let $\delta = \mu_r \alpha \varepsilon$. Let $G$ be an $n$-vertex graph with minimum degree at least $(\frac{2r-5}{2r-3}+\alpha)n$ and with at most $\delta n^r$ copies of $K_r$.

	Let $E^*$ denote the set of edges in $G$ which lie in at least $c_r \alpha n^{r-2}$ copies of $K_r$. Then the number of $K_r$ in $G$ is at least $\binom r2 \inv c_r \alpha n^{r-2} \ab{E^*}$, since each edge in $E^*$ contributes at least $c_r \alpha n^{r-2}$ copies, and we count each copy at most $\binom r2$ times (once for each edge). By assumption, $G$ has at most $\delta n^r$ copies of $K_r$, and combining these bounds, we find that
	\[
		\ab{E^*} \leq \frac{\binom r2}{c_r} \frac{\delta}{\alpha} n^2= \varepsilon n^2,
	\]
	by our choice of $\delta = \mu_r\alpha \varepsilon$.

	Additionally, by \cref{lem:key-lemma}, we know that every $K_r$ in $G$ contains at least one edge from $E^*$. Hence, if we delete the edges in $E^*$, we are left with a $K_r$-free graph. Since we deleted at most $\varepsilon n^2$ edges, this completes the proof.
\end{proof}

\section{Below the threshold: the proof of Theorem \ref{thm:below-threshold}}\label{sec:below-threshold}
In this section, we prove \cref{thm:below-threshold} by constructing a graph with high minimum degree, few copies of $K_r$, but such that many edges must be removed to make it $K_r$-free. Formally, we will prove the following result.
\begin{thm}\label{thm:below-threshold-extended}
	For every integer $r \geq 3$, parameters $\alpha>0$ and $\varepsilon>0$, and all sufficiently large $n$, there exists an $n$-vertex graph $G$ with minimum degree at least $(\frac{2r-5}{2r-3}- \alpha)n$ and with at most $\frac{\alpha^3}{(r/3)^r} \delta ( \frac{(2r-3)^2}{\alpha^2}\varepsilon, K_3 ) n^r$ copies of $K_r$, but at least $\varepsilon n^2$ edges must be deleted from $G$ to make it $K_r$-free. Therefore, 
	\[
		\delta\left(\varepsilon,K_r;\frac{2r-5}{2r-3}- \alpha\right) \leq \frac{\alpha^3}{(r/3)^r} \delta \left( \frac{(2r-3)^2}{\alpha^2}\varepsilon, K_3 \right).
	\]
\end{thm}
Note that \cref{thm:below-threshold-extended} is somewhat stronger than \cref{thm:below-threshold}, because we discarded the factor $\frac{\alpha^3}{(r/3)^r}$ in the statement of \cref{thm:below-threshold}. We will first prove \cref{thm:below-threshold-extended} in the case $r=3$, and then show how to extend this construction to prove \cref{thm:below-threshold-extended} for all $r \geq 4$. 

We will need two simple and well-known lemmas. The first says that balanced blowups preserve the triangle removal properties of graphs.
\begin{lem}\label{lem:blowup-preserves}
	Let $H_0$ be a graph with $t$ triangles, let $m$ be the minimum number of edges that one can delete to make $H_0$ triangle-free, and let $s$ be a positive integer. Then the balanced blowup $H_0[s]$ has exactly $s^3t$ triangles, and the minimum number of edges one can delete to make $H_0[s]$ triangle-free is $s^2m$.
\end{lem}
\begin{proof}
	The first claim is immediate since every triangle in $H_0$ corresponds to $s^3$ triangles in $H_0[s]$. For the second, suppose we delete fewer than $s^2 m$ edges from $H_0[s]$. We pick a random copy of $H_0$ in $H_0[s]$ by independently picking a uniformly random vertex in each part of $H_0[s]$. Then the expected number of deleted edges in this copy of $H_0$ is less than $(s^2m)/s^2 = m$. So there exists a copy of $H_0$ in $H_0[s]$ with fewer than $m$ edges deleted, and this copy must contain a triangle by assumption. Thus, by deleting fewer than $s^2 m$ edges, we cannot destroy all triangles in $H_0[s]$. On the other hand, suppose we are given a set $E^*$ of $m$ edges in $H_0$ whose deletion destroys all triangles in $H_0$. By deleting all $s^2 m$ edges of $H_0[s]$ which correspond to a blown-up copy of an edge in $E^*$, we delete all triangles in $H_0[s]$, so $s^2 m$ edge deletions suffice to destroy all triangles in $H_0[s]$.
\end{proof}

The second lemma says that one can convert any construction for the triangle removal lemma into a tripartite construction with similar parameter dependencies. We remark that this lemma is not fully optimized, and one could obtain better constants through a more careful argument. 
\begin{lem}\label{lem:make-tripartite}
	Suppose that $H_0$ is an $n_0$-vertex graph with at most $\delta n_0^3$ triangles, but such that at least $\varepsilon n_0^2$ edges must be deleted to make $H_0$ triangle-free. Then there exists a tripartite graph $H$ on $n\coloneqq 3n_0$ vertices with at most $\frac 29 \delta n^3$ triangles, such that at least $\frac 19 \varepsilon n^2$ edges must be deleted to make $H$ triangle-free.
\end{lem}
\begin{proof}
	Consider the tensor product $H = H_0 \times K_3$, which is the graph whose vertices are pairs $(v,x) \in V(H) \times [3]$, and where two vertices $(v,x)$ and $(w,y)$ are adjacent if and only if $x \neq y$ and $v \sim w$ in $H_0$. We claim that $H$ has the desired properties.

	Indeed, by definition, $H$ is tripartite and has $n=3n_0$ vertices. Moreover, each triangle in $H_0$ yields precisely six triangles in $H$, so $H$ has at most $6 \delta n_0^3 =\frac 29 \delta n^3$ triangles. To conclude, suppose that $E^* \subseteq E(H)$ is a set of edges whose deletion makes $H$ triangle-free. Let $E^*_0 \subseteq E(H_0)$ denote the edges of $H_0$ obtained by deleting the second coordinate of every vertex in every edge of $E^*$; in particular, $\ab{E^*_0} \leq \ab{E^*}$. We claim that deleting the edges in $E_0^*$ makes $H_0$ triangle-free. Indeed, if $\{v_1,v_2,v_3\} \subseteq V(H_0)$ form a triangle in $H_0$ after the deletion of $E_0^*$, then we see that no edge in $E^*$ can be of the form $\{(v_i,x),(v_j,y)\}$ for any $i \neq j$ and $x \neq y$. In particular, we find that $\{(v_1,1),(v_2,2),(v_3,3)\}$ is a triangle in $H$ whose edges do not intersect $E^*$, contradicting the assumption that the deletion of $E^*$ destroyed all triangles in $H$. Hence, by the defining property of $H_0$, we conclude that
	\[
		\ab{E^*} \geq \ab{E_0^*} \geq \varepsilon n_0^2 = \frac 19 \varepsilon n^2,
	\]
	as claimed. 
\end{proof}

The next lemma is simply a restatement of  \cref{thm:below-threshold-extended} in the case $r=3$. We state it as a separate lemma because the $r=3$ construction will be used as a black box in the construction for larger values of $r$.
\begin{lem}\label{lem:below-threshold-r=3}
	For all $\alpha>0$ and $\varepsilon>0$ and all sufficiently large $n$, there exists a tripartite $n$-vertex graph $G_0$ with minimum degree at least $(\frac 13 - \alpha)n$ and with at most $\alpha^3\delta(9 \varepsilon/\alpha^2, K_3)n^3$ triangles, but at least $\varepsilon n^2$ edges must be deleted from $G_0$ to make it triangle-free.
\end{lem}
\begin{proof}
	We first claim that for all sufficiently large $n$, there exists an $n$-vertex graph with at most $2 \delta(\varepsilon,K_3)n^3$ triangles such that at least $\varepsilon n^2$ edges must be deleted to make it triangle-free. Indeed, by the definition of $\delta(\varepsilon,K_3)$, there is a sequence of graphs which are all $\varepsilon$-far from being triangle-free, but whose triangle density approaches $\delta(\varepsilon,K_3)$. Thus, there must exist a graph $H_0$ on some fixed number $n_0$ of vertices with at most $\frac 32 \delta(\varepsilon,K_3)n_0^3$ triangles such that at least $\varepsilon n_0^2$ edges must be deleted to make it triangle-free. 
	By \cref{lem:blowup-preserves}, for any $s\geq 1$, the balanced blowup $H_0[s]$ will have at most $\frac 32 \delta(\varepsilon,k_3)(sn_0)^2$ triangles, but at least $\varepsilon (sn_0)^2$ edges must be deleted to make $H_0[s]$ triangle-free.
	Therefore, for $n$ sufficiently large relative to $n_0$, we may take the blowup $H_0[\lfloor n/n_0\rfloor]$ and add to it $n-n_0\lfloor n/n_0\rfloor$ isolated vertices to obtain the desired graph.

	Therefore, by \cref{lem:make-tripartite} applied with parameters $\varepsilon' =9\varepsilon/\alpha^2$ and $n' = \alpha n/3$, there exists a tripartite graph $H$ on $\alpha n$ vertices with fewer than $\frac 29 \cdot 2 \delta(\varepsilon',K_3)(\alpha n)^3< \delta(\varepsilon',K_3)(\alpha n)^3$ triangles such that at least $\frac 19 \varepsilon' (\alpha n)^2$ edges must be deleted to make $H$ triangle-free. Let $\Gamma$ be a balanced blowup of the path with two edges, blown up so that it has $(1- \alpha)n$ vertices. Let $G_0$ be the graph obtained by taking the disjoint union of $H$ and $\Gamma$, and placing a complete bipartite graph between the $i$th part of $H$ and the $i$th part of $\Gamma$, for $i \in [3]$. Then $G_0$ is tripartite by definition. The construction is shown in \cref{fig:r=3-construction}.

	Since every vertex in $G_0$ is adjacent to all vertices in at least one part of $\Gamma$, we see that every vertex in $G_0$ has degree at least $\frac 13 (1- \alpha)n>(\frac 13 -\alpha)n$. Moreover, we see that every triangle in $G_0$ is actually contained in $H$, so the number of triangles in $G_0$ is at most $\delta( \varepsilon', K_3) (\alpha n)^3 = \alpha^3 \delta(9 \varepsilon/\alpha^2, K_3) n^3$. Finally, if we delete some edges to make $G_0$ triangle-free, we must in particular make $H$ triangle-free. Therefore, the number of edges needed to make $G_0$ triangle-free is at least $\frac 19 \varepsilon' (\alpha n)^2 = \frac 19 \cdot \frac{9\varepsilon}{\alpha^2} (\alpha n)^2 = \varepsilon n^2$.\qedhere

	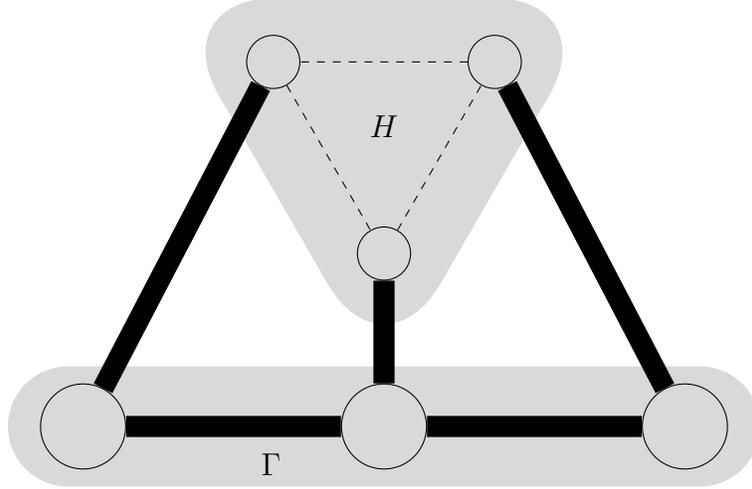
\begin{figure}[ht]
		\begin{center}
			\begin{tikzpicture}
				\node[draw, circle, inner sep=4mm] (A) at (-4, -4) {};
				\node[draw, circle, inner sep=4mm] (B) at (0, -4) {};
				\node[draw, circle, inner sep=4mm] (C) at (4, -4) {};
				\node[draw, circle, inner sep=2.5mm] (X) at (150:1.7) {};
				\node[draw, circle, inner sep=2.5mm] (Y) at (270:1.7) {};
				\node[draw, circle, inner sep=2.5mm] (Z) at (30:1.7) {};
				\draw[line width = 8pt] (A) -- (B);
				\draw[line width = 8pt] (B) -- (C);
				\draw[line width = 8pt] (Z) -- (C);
				\draw[line width = 8pt] (B) -- (Y);
				\draw[line width = 8pt] (A) -- (X);
				\draw[dashed] (X) -- (Y) -- (Z) -- (X);
				\node at (0,0) {$H$};
				\path (B) -- +(-1.5,-.5) node {$\Gamma$};
				\begin{scope}[on background layer]
					\fill[gray!30, rounded corners=15mm] (150:3.4) -- (270:3.4) -- (30:3.4) -- cycle;
					\fill[gray!30, rounded corners=8mm] (-5,-4.8) -- (5,-4.8) -- (5,-3.2) -- (-5, -3.2) -- cycle;
				\end{scope}
			\end{tikzpicture}
		\end{center}
		\caption{The construction in \cref{lem:below-threshold-r=3}. Solid edges represent complete bipartite graphs, and dashed edges represent the edges of $H$ as given by \cref{lem:make-tripartite}.} \label{fig:r=3-construction}
	\end{figure}
\end{proof}

With this result, we are now ready to prove \cref{thm:below-threshold-extended}, and thus \cref{thm:below-threshold}.
\begin{proof}[Proof of \cref{thm:below-threshold-extended}]
	If $r=3$, then the result is precisely the statement of \cref{lem:below-threshold-r=3}. So we henceforth assume that $r \geq 4$. Let $G_0$ be the graph from \cref{lem:below-threshold-r=3}, applied with parameters $\alpha, \varepsilon' = \left(\frac{2r-3}{3}\right)^2\varepsilon$, and $n' = \frac{3}{2r-3}n$. Let $K$ be a complete $(r-3)$-partite graph where each part has $\frac{2}{2r-3}n$ vertices, and let $G$ be the join of $G_0$ and $K$, i.e.\ the graph obtained from the disjoint union of $G_0$ and $K$ by connecting every vertex in $K$ to every vertex in $G_0$. Then $G$ has $(r-3)\frac{2}{2r-3}n + \frac{3}{2r-3}n = n$ vertices. In $G$, every vertex coming from $K$ has degree $(r-4)\frac{2}{2r-3}n+\frac{3}{2r-3}n = \frac{2r-5}{2r-3}n$, and every vertex coming from $G_0$ has degree at least
	\[
		(r-3)\frac{2}{2r-3}n + \left( \frac 13 - \alpha \right) \frac{3}{2r-3}n > \left( \frac{2r-5}{2r-3}- \alpha \right) n,
	\]
	hence $G$ has the desired minimum degree condition. Additionally, since $G_0$ is tripartite, it is $K_t$-free for all $t \geq 4$. Therefore, we see that every $K_r$ in $G$ must consist of a triangle in $G$ and $r-3$ vertices from $K$. Hence, the number of $K_r$ in $G$ is at most 
	\[
		\alpha^3 \delta\left(\frac{9 \varepsilon'}{\alpha^2}, K_3\right) (n')^3 \cdot \left( \frac{2n}{2r-3} \right) ^{r-3} \leq \frac{\alpha^3}{(r/3)^r} \delta \left( \frac{(2r-3)^2}{\alpha^2}\varepsilon, K_3 \right) n^r.
	\]
	Moreover, if we delete some edges to make $G$ be $K_r$-free, we must in particular make $G_0$ triangle-free. Thus, the number of edges that must be deleted is at least $\varepsilon' (n')^2=\varepsilon n^2$.
\end{proof}

\section{Concluding remarks}\label{sec:conclusion}
\subsection{The removal thresholds for other graphs}
Recall the definition of the linear and polynomial removal thresholds from \cref{def:removal-thresholds}. In this subsection, we make some remarks about the values of $\linrem(H)$ and $\polyrem(H)$ for more general classes of graphs. 

We begin by observing, directly from the definition, that $\polyrem(H) \leq \linrem(H)$ for any $H$, since a linear bound on the removal lemma is in particular a polynomial bound. Moreover, if $\ex(n,H)$ denotes the \emph{extremal number} of $H$, that is the maximum number of edges in an $H$-free graph on $n$ vertices, and if $\pi(H)\coloneqq\lim_{n \to \infty}\binom n2 \inv \ex(n,H)$ is the \emph{Tur\'an density} of $H$, then we have that $\linrem(H)\leq \pi(H)$. Indeed, this follows from the Erd\H os--Simonovits supersaturation theorem \cite{MR726456}, which implies that for any $\alpha>0$, there exists some $\delta_0>0$ such that every $n$-vertex graph $G$ with minimum degree at least $(\pi(H)+\alpha)n$ has at least $\delta_0 n^{\ab{V(H)}}$ copies of $H$. This shows that $\delta(\varepsilon,H;\pi(H)+\alpha)\geq \delta_0$ for all $\varepsilon>0$, which in turn implies that $\linrem(H)\leq \pi(H)+\alpha$ by taking $\mu=\delta_0$ in \cref{def:removal-thresholds}. Letting $\alpha$ tend to zero yields that $\linrem(H) \leq \pi(H)$.

Our first result in this section shows that $\polyrem$ is invariant under a certain natural relation on graphs. Recall that a graph homomorphism $H_2 \to H_1$ is a function $V(H_2) \to V(H_1)$ that maps every edge of $H_2$ to an edge of $H_1$.
\begin{prop}\label{prop:core-invariance}
	If $H_1$ is a subgraph of $H_2$ and there exists a homomorphism $H_2 \to H_1$, then
	\[
		\polyrem(H_1) = \polyrem(H_2).
	\]
\end{prop}
\begin{proof}
	Let $h_1 = \ab{V(H_1)}, h_2 = \ab{V(H_2)}$, and fix a graph homomorphism $\varphi:H_2 \to H_1$. Let $s_1,\dotsc,s_{h_1}$ be the sizes of the fibers of $\varphi$, i.e.\ $s_i = \ab{\varphi\inv(v_i)}$ for $v_i \in V(H_1)$. Let $s = \max \{s_1,\dotsc,s_{h_1}\}$, so that $H_2$ is a subgraph of the blowup $H_1[s]$. 

	We first prove that $\polyrem(H_1) \leq \polyrem(H_2)$. For this, let $G$ be an $n$-vertex graph with minimum degree $\gamma n$ that has at most $\delta n^{h_1}$ copies of $H_1$, such that at least $\varepsilon n^2$ edges must be removed from $G$ to make it $H_1$-free. By the same argument as in the first paragraph of \cref{lem:below-threshold-r=3}, we may assume without loss of generality that $n$ is sufficiently large. The number of non-injective homomorphisms $H_1 \to G$ is at most $\binom{h_1}{2} n^{h_1-1}$, which is at most $\delta n^{h_1}$ for $n$ sufficiently large. So the number of homomorphisms $H_1 \to G$ is at most $2 \delta n^{h_1}$, which implies that the number of copies of $H_1$ in the blowup $G[s]$ is at most $2 \delta (sn)^{h_1}$. Every copy of $H_1$ in $G[s]$ can be extended to a copy of $H_2$ in at most $(sn)^{h_2-h_1}$ ways, which implies that $G[s]$ contains at most $2 \delta(sn)^{h_2}$ copies of $H_2$. An $H_2$-free subgraph $\Gamma \subseteq G[s]$ yields an $H_1$-free subgraph of $G$ by keeping those edges of $G$ all of whose lifts are present in $\Gamma$; any copy of $H_1$ in this subgraph would lift to a copy of $H_1[s] \supseteq H_2$ in $\Gamma$, a contradiction. This implies that at least $\varepsilon n^2 = \frac \varepsilon{s^2}(sn)^2$ edges must be deleted from $G[s]$ to make it $H_2$-free. Since $G[s]$ has minimum degree $\gamma(sn)$ and $sn$ vertices, we conclude that
	\[
		\delta(\varepsilon, H_1; \gamma) \geq \frac 12 \delta \left( \frac \varepsilon{s^2}, H_2; \gamma \right) ,
	\]
	and therefore that $\polyrem(H_1) \leq \polyrem(H_2)$. 

	For the reverse inequality, now let $G$ be an $n$-vertex graph with minimum degree $\gamma n$ that has at most $\delta n^{h_2}$ copies of $H_2$, such that at least $\varepsilon n^2$ edges must be removed from $G$ to make it $H_2$-free, and we again assume that $n$ is sufficiently large. Since an $H_1$-free subgraph of $G$ is also $H_2$-free, we see that at least $\varepsilon n^2$ edges must be removed from $G$ to make it $H_1$-free. Let $m = s_1\dotsb s_{h_1}$. We claim that $G$ has at most $3h_1^{h_1} \delta^{1/m} n^{h_1}$ copies of $H_1$. For if not, then we can randomly partition $V(G)$ into $h_1$ sets to obtain an $h_1$-partite subgraph $G'$ with at least $3 \delta^{1/m} n^{h_1}$ canonical copies of $H_1$, where we say that a copy is \emph{canonical} if the $i$th vertex of $H_1$ lies in the $i$th part of $G'$ for all $i \in [h_1]$. Let $\h$ be the $h_1$-uniform hypergraph on $V(G)$ whose edges are the canonical copies of $H_1$ in $G'$, so that $\h$ has edge density at least $\eta \coloneqq 3 \delta^{1/m}$. An argument of Erd\H os \cite{MR183654} (see also \cite{HypergraphSidorenko}) implies that if $\K$ is a complete $h_1$-partite $h_1$-uniform hypergraph with $m$ edges, then there at least $\eta^m n^{\ab{V(\K)}}$ homomorphisms $\K \to \h$. For $n$ sufficiently large, and taking $\K$ to have parts of size $s_1,\dotsc,s_{h_1}$, we conclude that $\h$ has at least $\frac 12 \eta^m n^{h_2} > \delta n^{h_2}$ copies of $\K$. This implies that $G'$ contains more than $\delta n^{h_2}$ copies of $H_2$, a contradiction. We conclude that
	\[
		\delta(\varepsilon, H_2; \gamma) \geq \frac{1}{3h_1^{h_1}} \delta(\varepsilon, H_1;\gamma)^m,
	\]
	and therefore that $\polyrem(H_2) \leq \polyrem(H_1)$. 
\end{proof}

Using \cref{prop:core-invariance}, we can determine the polynomial removal threshold of many graphs. For instance, if $H$ is a non-bipartite graph whose clique number $\omega(H)$ equals its chromatic number $\chi(H)$, then
\[
	\polyrem(H) = \frac{2 \omega(H)-5}{2\omega(H)-3},
\]
since $K_{\omega(H)}$ is a subgraph of $H$ and there is a homomorphism $H \to K_{\omega(H)}$. In particular, we are able to determine the polynomial removal threshold of all perfect graphs.

Additionally, \cref{prop:core-invariance} implies that $\polyrem(H)=\polyrem(H[s])$ for any graph $H$ and any integer $s\geq 1$, that is that the polynomial removal threshold is invariant under blowups. More generally, we recall that the \emph{core} of a graph $H$ is defined as the inclusion-minimal subgraph $K$ such that there exists a homomorphism $H \to K$; see e.g.\ \cite{MR1192374} for more on this concept. Then by \cref{prop:core-invariance}, we see that $\polyrem(H)=\polyrem(K)$ for any graph $H$ and its core $K$. In other words, the polynomial removal threshold of a graph is completely determined by that of its core.

In contrast to the above results, the linear removal threshold does not satisfy such a nice invariance property. Indeed, our next result demonstrates that for any $r \geq 3$ and $s \geq 2$, the complete multipartite graph $K_r[s]$ has linear removal threshold $\frac{r-2}{r-1}$; this equals the Tur\'an density $\pi(K_r[s])$, and is strictly larger than $\linrem(K_r)=\frac{2r-5}{2r-3}$.
\begin{prop}\label{prop:complete-multipartite}
	For any $r\geq 3$ and $s \geq 2$,
	\[
		\linrem(K_r[s]) = \frac{r-2}{r-1}.
	\]
\end{prop}
\begin{proof}
	As remarked above, the inequality $\linrem(K_r[s])\leq \frac{r-2}{r-1}$ follows from the Erd\H os--Stone theorem \cite{MR18807}, which says that $\pi(K_r[s])=\frac{r-2}{r-1}$. For the reverse inequality, it suffices to construct an $n$-vertex graph $G$ with minimum degree at least $\frac{r-2}{r-1}n$ and fewer than $\delta n^{rs}$ copies of $K_r[s]$, but such that at least $\varepsilon n^2$ edges must be deleted to make it $K_r[s]$-free, where $\varepsilon$ cannot be taken to depend linearly on $\delta$. We may assume throughout that $n$ is sufficiently large.

	To do so, we let $T(n,r-1)$ denote the Tur\'an graph, that is the complete $(r-1)$-partite graph with parts of size $\ab{S_1}=\dotsb=\ab{S_{r-1}}=\frac{n}{r-1}$ (where we assume for simplicity that $2(r-1)$ divides $n$). Inside the part $S_1$ of $T(n,r-1)$, we place a random graph $G(\frac{n}{r-1},p)$ for some fixed $p \in (0,\frac{1}{4s^2})$; in other words, we connect every pair in $S_1$ by an edge with probability $p$, independently over all these choices. We let $G$ be the resulting graph. Then we immediately see that $G$ has minimum degree at least $\frac{r-2}{r-1}n$, since that was the case in $T(n,r-1)$. For clarity, we now fix $G[S_1]$ to be a graph where every vertex has degree $(p+o(1))\ab{S_1}$ and the $K_{s,s}$ density in $G[S_1]$ is $p^{s^2}+o(1)$. This is possible since both these properties hold in $G(\frac{n}{r-1},p)$ with high probability. Note that every copy of $K_r[s]$ in $G$ must contain a copy of $K_{s,s}$ in $S_1$. Therefore, $G$ contains at most $\delta n^{rs}$ copies of $K_r[s]$, where $\delta =p^{s^2}$. 

	Now, suppose that $G'$ is a $K_r[s]$-free subgraph of $G$ with the maximum possible number of edges. Recall that in $G$, every vertex of $S_1$ has degree at most $2p \frac{n}{r-1}$ inside $S_1$. Therefore, if any vertex $v \in S_1$ has more than $2p \frac{n}{r-1}$ non-neighbors in $S_2 \cup \dotsb \cup S_{r-1}$, we can find a $K_r[s]$-free subgraph of $G$ with more edges than $G'$ by deleting all edges incident to $v$ in $S_1$ and adding all missing edges to $S_2 \cup \dotsb \cup S_{r-1}$. 

	We now show $G'[S_1]$ is $K_{s,s}$-free. Suppose not, and consider a copy of $K_{s,s}$ in $G'[S_1]$. Since every vertex of this $K_{s,s}$ has at most $2p \frac{n}{r-1}$ non-neighbors in $S_2\cup \dotsb \cup S_{r-1}$, we find that the vertices of the $K_{s,s}$ have in total at most $2s\cdot 2p \frac{n}{r-1}$ non-neighbors outside $S_1$. As $2s \cdot 2p \frac{n}{r-1} \leq \frac 12 \ab{S_i}$ since $p<1/(4s^2) \leq 1/(8s)$, we see that there are subsets $S_i' \subseteq S_i$ for $2 \leq i \leq r-1$ such that $\ab{S_i'} = \frac 12 \ab{S_i}$ and every vertex in $S_i'$ is complete to the vertices in the $K_{s,s}$. If $r=3$, we arrive at a contradiction as $\ab{S_2'} \geq s$ for $n$ sufficiently large, so the $K_{s,s}$ together with $s$ vertices from $S_2'$ yields a copy of $K_3[s]$. We now assume $r \geq 4$. As $G'$ is $K_r[s]$-free, $G'[S_2' \cup \dotsb \cup S_{r-1}']$ does not contain a copy of $K_{r-2}[s]$. 
	Now suppose we sample $s$ random vertices from each $S_i'$. These $(r-2)s$ vertices cannot span a copy of $K_{r-2}[s]$, so these vertices must include both endpoints of at least one edge that is present in $G$ but not $G'$. On the other hand, if there are $q_{ij}\ab{S_i'}\ab{S_j'}$ such deleted edges of between $S_i'$ and $S_j'$ for all $2 \leq i \neq j \leq r-1$, then the expected number of deleted edges among the $(r-2)s$ sampled vertices is $s^2 \sum_{i \neq j} q_{ij}$. As this expectation must be at least $1$, we conclude that the number of edges deleted when passing from $G$ to $G'$ is at least $\sum_{i \neq j} q_{ij} \ab{S_i'}\ab {S_j'} = \frac{1}{4}(\frac{n}{r-1})^2 \sum_{i \neq j}q_{ij} \geq \frac{1}{4s^2} (\frac{n}{r-1})^2$.
	However, this is more than the number of edges in $G[S_1]$, so the graph obtained from $G$ by deleting all edges inside $S_1$ has more edges than $G'$ and is also $K_r[s]$-free, a contradiction. Thus, $G'[S_1]$ is $K_{s,s}$-free.

	Any $K_{s,s}$-free subgraph of $G[S_1]$ has at most $O(n^{2-1/s})$ edges by the K\H ov\'ari--S\'os--Tur\'an theorem \cite{MR65617}, so we conclude that at least $\frac 12 p \binom{\ab{S_1}}{2}\geq  \varepsilon n^2$ edges must have been deleted when going from $G$ to $G'$, where $\varepsilon = \frac{p}{4r^2}$. Since $G$ has at most $\delta n^{rs}$ copies of $K_r[s]$, where $\delta = p^{s^2}$, we see that the dependence between $\varepsilon$ and $\delta$ cannot be linear. 
\end{proof}

To conclude this subsection, we turn our attention to cycles. Since an even cycle is bipartite, its Tur\'an density is $0$, and hence so are its linear and polynomial removal thresholds. For odd cycles, we are able to prove the following lower bound, though we do not know if it is tight. The construction is a natural and simple generalization of that in \cref{lem:below-threshold-r=3}, which corresponds to the case $k=1$ in the following theorem. 

\begin{thm}\label{thm:odd-cycle}
	For every positive integer $k$, we have that $\polyrem(C_{2k+1})\geq \frac{1}{2k+1}$. 
\end{thm}
\begin{proof}
	For every $k \geq 1$, every sufficiently small $\varepsilon_0>0$, and every sufficiently large $m$, Alon \cite{MR1945375} constructed a graph $H$ on $m$ vertices with vertex set $V_0 \sqcup \dotsb \sqcup V_{2k}$ with the following properties. Every edge in $H$ goes between $V_i$ and $V_{i+1}$ for some $i$ (with the indices taken modulo $2k+1$), at least $\varepsilon_0 m^2$ edges must be removed from $H$ to make it $C_{2k+1}$-free, and $H$ has at most $\delta m^{2k+1}$ copies of $C_{2k+1}$, where $\delta = {\varepsilon_0}^{-c \log {\varepsilon_0}}$,
	for some constant $c>0$ depending only on $k$. 

	We set $\varepsilon_0=\varepsilon/\alpha^2$ and $m=\alpha n$ and adjoin to this graph $H$ a balanced blowup $\Gamma$ of the path on $2k+1$ vertices, blown up so that $\Gamma$ has $(1- \alpha)n$ vertices in total. Finally, we place a complete bipartite graph between the $i$th part of $\Gamma$ and the $i$th part of $H$. Then the resulting graph has minimum degree at least $(\frac 1{2k+1} -\alpha)n$, and every $C_{2k+1}$ in this graph is contained in $H$. Therefore, this resulting graph has at most $\delta m^{2k+1}\leq \delta n^{2k+1}$ copies of $C_{2k+1}$, but at least $\varepsilon_0 m^2 = \varepsilon n^2$ edges must be removed to make it $C_{2k+1}$-free. Since $1/\delta$ is super-polynomial in $1/\varepsilon$, this gives the theorem. 
\end{proof}

\subsection{The popular edge threshold}
Our proof of \cref{thm:above-threshold} used \cref{lem:key-lemma}, which is a very natural way of proving linear bounds on the removal lemma. Recall that \cref{lem:key-lemma} says that if $G$ has minimum degree at least $(\frac{2r-5}{2r-3}+\alpha)n$, then every copy of $K_r$ in $G$ has a ``popular'' edge, namely an edge that lies in $\Omega (n^{r-2})$ copies of $K_r$. Given this statement, the proof of \cref{thm:above-threshold} is simple, since we simply delete a popular edge from each copy of $K_r$, which necessarily yields linear bounds for the $K_r$ removal lemma.

This discussion naturally leads to the following definition.
\begin{Def}
	Let $H$ be a graph. The \emph{popular edge threshold} of $H$ is defined as the infimum of all $\gamma \in [0,1]$ for which the following holds. There exists $\beta = \beta(\gamma)>0$ such that for every $n$-vertex graph $G$ with minimum degree at least $\gamma n$, every copy of $H$ in $G$ contains an edge which lies in at least $\beta n^{\ab{V(H)}-2}$ copies of $H$.
\end{Def}
From the same argument as before, we see that $\linrem(H) \leq \popedge(H)$, and \cref{lem:key-lemma} shows that $\popedge(K_r) \leq \frac{2r-5}{2r-3}$, which is a tight bound by \cref{thm:below-threshold}. However, in general, $\popedge(H)$ can be strictly larger than $\linrem(H)$, as shown in the following proposition. 
\begin{prop}\label{prop:popedge=0}
	Let $H$ be a graph with no isolated vertices. Then $\popedge(H)=0$ if and only if $H$ is bipartite and every edge of $H$ lies in a cycle.
\end{prop}
Recall that if $H$ is bipartite, then $\pi(H)=0$ and therefore $\linrem(H)=0$ as well. Thus, every bipartite graph in which every edge lies in a cycle is an example of a graph whose popular edge threshold is strictly larger than its linear removal threshold.
\begin{proof}[Proof of \cref{prop:popedge=0}]
	Let $h = \ab{V(H)}$. First suppose that $H$ is not bipartite or every edge of $H$ lies in a cycle. For any integer $s \geq 1$, consider the blowup $C_{h^2}[s]$, where we label the parts $0,1,\dotsc,h^2-1$. Form a graph $G$ by adding to $C_{h^2}[s]$ a single copy of $H$, with one vertex in each of the parts labeled $0,h,2h,\dotsc,h^2-h$. Then $G$ has $n\coloneqq h^2s$ vertices and minimum degree $2s = \frac{2}{h^2}n$. If $H$ is not bipartite, then the only odd cycles of length at most $h$ in $G$ are in the added copy of $H$. Since any odd cycle can be extended to a copy of $H$ in at most $O_H(n^{h-3})$ ways, we conclude that $G$ has at most $O_H(n^{h-3})$ copies of $H$. For any fixed $\beta>0$, by letting $s$ (and thus $n$) be sufficiently large, this implies that $G$ contains fewer than $\beta n^{h-2}$ copies of $H$, and in particular the added copy of $H$ has no popular edge. Similarly, if every edge of $H$ lies in a cycle, then in particular every edge lies in a cycle of length at most $h$. For any edge in the added copy of $H$, the only cycles of length at most $h$ it participates in are in the added copy, so any such edge appears in at most $O_H(n^{h-4})$ copies of $H$, again showing that this copy of $H$ has no popular edge. This implies that $\popedge(H) \geq \frac{2}{h^2}>0$.

	For the reverse implication, suppose that $H$ is bipartite and has an edge which lies in no cycle. Let $u_1 u_2 \in E(H)$ be such an edge, and note that its deletion makes $H$ disconnected. Let $H_1,H_2$ be induced subgraphs of $H$ containing $u_1$ and $u_2$, respectively, so that the only edge between $H_1$ and $H_2$ is $u_1 u_2$. For $i \in \{1,2\}$, let the bipartition of $H_i$ be $V(H_i) = A_i \cup B_i$, with $u_i \in A_i$. Let $\ab{A_i} = a_i, \ab{B_i}=b_i$, so that $h=a_1+b_1+a_2+b_2$. Let $\gamma>0$, and let $G$ be a graph with $n$ vertices and minimum degree at least $\gamma n$. We claim that for all positive integers $a,b$, and for any vertex $v \in V(G)$, there are at least $\Omega_{\gamma,a,b}(n^{a+b-1})$ copies of $K_{a, b}$ in $G$ which contain $v$ as one of the $a$ vertices in the first part. This follows from the supersaturation version \cite{MR726456} of the K\H ov\'ari--S\'os--Tur\'an theorem \cite{MR65617} on the problem of Zarankiewicz, which proves the existence of many copies of $K_{a-1,b}$ in the auxiliary bipartite graph with parts $V(G) \setminus \{v\}$ and $N(v)$, whose edges are given by adjacency in $G$. This implies that for any edge $v_1 v_2$ of $G$, there are at least $\Omega_{\gamma,a_1,b_1,a_2,b_2}(n^{h-2})$ copies of $H$ in $G$ containing the edge $v_1 v_2$, since $H$ is a subgraph of the graph obtained from the disjoint union of $K_{a_1,b_1}$ and $K_{a_2,b_2}$ by adding a single edge between $A_1$ and $A_2$.
	
	This shows that every edge in $G$ lies in at least $\Omega_{\gamma,H}(n^{h-2})$ copies of $H$. In particular, in any copy of $H$ in $G$, any edge is a popular edge, showing that $\popedge(H) \leq \gamma$. Letting $\gamma \to 0$ gives the desired result.
\end{proof}

This example demonstrates that our approach to upper-bounding the linear removal threshold via the popular edge threshold will not give tight bounds in general. Nevertheless, we think that it is interesting to study when this approach will yield a tight bound, i.e.\ to understand when $\linrem(H) = \popedge(H)$.
\begin{question}
	For which graphs $H$ does $\linrem(H) = \popedge(H)$? 
\end{question}

\subsection{The chromatic and homomorphism thresholds}
The number $\frac{2r-5}{2r-3}$, which emerges from \cref{thm:above-threshold,thm:below-threshold} as the linear and polynomial removal threshold of $K_r$, is a well-known number in extremal graph theory. Indeed, it turns out that $\frac{2r-5}{2r-3}$ is also both the \emph{chromatic threshold} and \emph{homomorphism threshold} of $K_r$. These are defined as follows. For a family of graphs $\F$ and a parameter $\gamma \in [0,1]$, let $\G(\F,\gamma)$ denote the set of $\F$-free graphs $G$ with minimum degree at least $\gamma \ab{V(G)}$. 
\begin{Def}
	Let $\F$ be a family of graphs. The \emph{chromatic threshold} of $\F$ is the number
	\[
		\delta_\chi(\F) = \inf \{\gamma \in [0,1]: \text{there exists }M>0 \text{ such that }\chi(G) \leq M\text{ for all } G \in \G(\F, \gamma)\}
	\]
	and the \emph{homomorphism threshold} of $\F$ is
	\begin{align*}
		\delta_{\hom}(\F) = \inf \{\gamma \in [0,1]:\text{there }&\text{exists an }\F\text{-free graph }G_0\text{ such that for all }G \in \G(\F,\gamma),\\
		&G\text{ has a homomorphism to }G_0\}.
	\end{align*}
	If $\F=\{F\}$ consists of a single graph, we denote these by $\delta_\chi(F)$ and $\delta_{\hom}(F)$.
\end{Def}
In other words, the chromatic threshold measures what minimum degree conditions force an $\F$-free graph to have a homomorphism to a graph of bounded order, and the homomorphism threshold further requires this bounded graph to itself be $\F$-free. Due to the efforts of many researchers \cite{MR342429,MR1956996,1001.2070,MR2260851,MR2791450,MR4152563}, it is now known that
\[
	\delta_\chi(K_r) = \delta_{\hom}(K_r) = \frac{2r-5}{2r-3}.
\]
Despite the fact that we get the same answer for all cliques, we were not able to find any \emph{a priori} relationship between the two removal thresholds and the chromatic or homomorphism thresholds. Moreover, such a relationship does not hold in general. For instance, Thomassen \cite{MR2321926} proved that $\delta_\chi(C_{2k+1}) = 0$ for all $k\geq 2$, while \cref{thm:odd-cycle} shows that $\linrem(C_{2k+1})\geq \polyrem(C_{2k+1})\geq \frac{1}{2k+1}$. Thus, the chromatic threshold of $C_{2k+1}$ is different from both removal thresholds. The homomorphism threshold of $C_{2k+1}$ is unknown for any $k\geq 2$, though Letzter--Snyder \cite{MR3879964} and Ebsen--Schacht \cite{MR4078811} proved that $\delta_{\hom}(C_{2k+1}) \leq \frac 1{2k+1}$, and that $\delta_{\hom}(\{C_3,C_5,\dotsc,C_{2k+1}\}) = \frac 1{2k+1}$. It would be very interesting to determine if in fact $\delta_{\hom}(C_{2k+1}) = \frac 1{2k+1}$, as well as what the values of $\polyrem(C_{2k+1})$ and $\linrem(C_{2k+1})$ are. As a first pass, we make the following conjecture.
\begin{conj}
	$\polyrem(C_{2k+1})>\delta_{\hom}(C_{2k+1})$ for all $k \geq 2$. 
\end{conj}

Odd cycles provide an example of graphs where the polynomial removal threshold is strictly larger than the chromatic threshold. In the other direction, we see from \cref{prop:core-invariance} that $\polyrem(K_3[2])=\polyrem(K_3)=\frac 13$, while it is well-known that $\delta_\chi(K_3[2]) = \frac 12$ (see \cite{MR3010059}, where this is stated as a folklore result). Thus, $K_3[2]$ (or more generally a non-trivial balanced blowup of a clique) is an example of a graph whose polynomial removal threshold is strictly smaller than its chromatic threshold. However, in this case, we have the curious situation that $\linrem(K_3[2])=\delta_\chi(K_3[2])$, by \cref{prop:complete-multipartite}. We conjecture that this is a coincidence, and that in general, the four thresholds have nothing to do with one another.
\begin{conj}
	There exists a graph $H$ for which $\delta_\chi(H), \delta_{\hom}(H), \linrem(H), \polyrem(H)$ are all distinct. 

	More generally, the numbers $\delta_\chi(H), \delta_{\hom}(H), \linrem(H), \polyrem(H)$ may appear in any order in $[0,1]$, subject to the constraints $\delta_\chi(H)\leq \delta_{\hom}(H)$ and $\polyrem(H)\leq \linrem(H)$. 
\end{conj}

To conclude, we remark that Allen, B\"ottcher, Griffiths, Kohayakawa, and Morris \cite{MR3010059} determined $\delta_\chi(H)$ for all graphs $H$. Moreover, they showed that if $\chi(H)=r\geq 3$, then $\delta_\chi(H)$ can take only one of three values, namely
\[
	\delta_\chi(H) \in \left\{ \frac{r-3}{r-2}, \frac{2r-5}{2r-3}, \frac{r-2}{r-1} \right\}.
\]
We are not bold enough to make any specific conjecture about the removal thresholds for arbitrary graphs. But we do leave the following open question, inspired by the Allen--B\"ottcher--Griffiths--Kohayakawa--Morris theorem.
\begin{question}
	Is it the case that for each $r \geq 3$, there exists a finite set $\Delta_r \subset [0,1]$ such that $\polyrem(H) \in \Delta_r$ for every graph $H$ with $\chi(H)=r$? What if we replace $\polyrem(H)$ by $\linrem(H)$, by $\popedge(H)$, or by $\delta_{\hom}(H)$?
\end{question}

\paragraph{Acknowledgments:} We would like to thank the anonymous referees for many helpful comments which greatly improved the presentation of this paper.

\end{document}